\documentclass[12pt]{amsart}
\usepackage{amsmath}
\usepackage{latexsym}
\usepackage{amssymb}
\usepackage{graphicx}
\DeclareMathOperator{\supp}{supp}

\DeclareMathOperator{\Char}{Char}

\DeclareMathOperator{\sing}{sing}
\DeclareMathOperator{\Span}{span}

\makeatletter
\@namedef{subjclassname@2020}{%
  \textup{2020} Mathematics Subject Classification}
\makeatother

\newtheorem{theorem}{Theorem}
\newtheorem{proposition}{Proposition}
\newtheorem{lemma}{Lemma}
\newtheorem{corollary}{Corollary}
\newtheorem{definition}{Definition}

\newtheorem{remark}{Remark}

\begin{document}
%
% \baselineskip 18pt
%%%%%%%%%%%%%%%%%%%%%%%%%%%%%%%%%%%%%%%%%%%%%%%%%%%%%%%%%%%%%%%%%%%%%%
%%%%%%%%%%%%%%%% Blackboard boldface common symbols  %%%%%%%%%%%%%%%%%
%%%%%%%%%%%%%%%%%%%%%%%%%%%%%%%%%%%%%%%%%%%%%%%%%%%%%%%%%%%%%%%%%%%%%%
\def\R {{\mathbb{R}}}
\def\N {{\mathbb{N}}}
\def\C {{\mathbb{C}}}
\def\Z {{\mathbb{Z}}}
%%%%%%%%%%%%%%%%%%%%%%%%%%%%%%%%%%%%%%%%%%%%%%%%%%%%%%%%%%%%%%%%%%%%%%
%%%%%%%%%%%%%%%%%% varsymbols ... for readability %%%%%%%%%%%%%%%%%%%%
%%%%%%%%%%%%%%%%%%%%%%%%%%%%%%%%%%%%%%%%%%%%%%%%%%%%%%%%%%%%%%%%%%%%%%
\def\phi{\varphi}
\def\epsilon{\varepsilon}
\def\ma{{\mathcal A}}
%%%%%%%%%%%%%%%%%%%%%%%%%%%%%%%%%%%%%%%%%%%%%%%%%%%%%%%%%%%%%%%%%%%%%%
%%%%%%%%%%%%%%%%%%%%%%%%%%%%%%%%%%%%%%%%%%%%%%%%%%%%%%%%%%%%%%%%%%%%%%
%
%%%%%%%%%%%%%%%%%%%%%%%%%%%%%%%%%%%%%%%%%%%%%%%%%%%%%%%%%%%%%%%%%%%%%%
%%%%%%%%%%%%%%%%%%%  math local macros %%%%%%%%%%%%%%%%%%%%%%%%%%%%%%
%%%%%%%%%%%%%%%%%%%%%%%%%%%%%%%%%%%%%%%%%%%%%%%%%%%%%%%%%%%%%%%%%%%%%%
\def\tb#1{\|\kern -1.2pt | #1 \|\kern -1.2pt |} 
\def\Qed{\qed\par\medskip\noindent}
%%%%%%%%%%%%%%%%%%%%%%%%%%%%%%%%%%%%%%%%%%%%%%%%%%%%%%%%%%%%%%%%%%%%%%
%%%%%%%%%%%%%%%%%%%%%%%%%%%%%%%%%%%%%%%%%%%%%%%%%%%%%%%%%%%%%%%%%%%%%%
%

\title[Lipschitz regularity of the minimum time function]{Partial Lipschitz regularity \\ of the minimum time function \\ for sub-Riemannian control systems} 
\author{Paolo Albano} 
\address{Dipartimento di Matematica, Universit\`a di Bologna, Piazza di Porta San Donato 5, 40127 Bologna, Italy} 
\email{paolo.albano@unibo.it}
 \author{Vincenzo Basco}
\address{Thales Alenia Space, via Saccomuro 24, Roma, Italy}
\email{vincenzo.basco@thalesaleniaspace.com,}
\email{vincenzobasco@gmail.com}
\author{Piermarco Cannarsa} 
\address{Dipartimento di Matematica, Universit\`a di Roma "Tor Vergata",  Roma, Italy} 
\email{cannarsa@mat.uniroma2.it}

\date{\today}

\begin{abstract}
  In Euclidean space of dimension 2 or 3, we study a minimum time problem associated with a system of  real-analytic  vector fields satisfying H\"ormander's bracket generating condition, where the target is a nonempty closed set.  We show that, in dimension 2, the minimum time function is locally Lipschitz continuous while, in dimension 3, it is Lipschitz continuous in the complement of a set of measure zero. In particular, in both cases, the minimum time function is a.e. differentiable on the complement of the target.  In dimension 3, in general, there is no hope to have the same regularity result as in dimension 2. Indeed, examples are known (see Remark 6.6 below) where the minimum time function fails to be locally Lipschitz continuous.
\end{abstract}


%\keywords{Hamiltonian systems, Sub-Riemannian geometry, Singular trajectories}

\maketitle

\section{Introduction}\
Let $f_1,\ldots ,f_m$ be  a family of  real-analytic vector fields  defined on $\R^n$, $n\geq 2$. Set 
$$
\mathcal U:=\{ u:[0,+\infty [\longrightarrow \bar{B}^m_1(0)\, |\, u\text{ is Lebesgue measurable}\},
$$
    where $\bar{B}^m_1(0)$ stands for the closed unit ball of $\R^m$ with center at the origin.
    Define
\begin{equation}\label{eq:f} 
f(x)u=\sum_{i=1}^m f_i(x) u_i.
  \end{equation}

    We consider the state equation  
\begin{equation}\label{eq:s0}
y' (t)=f(y(t))u(t)\quad \text{for a.e. }t\geq 0,
\end{equation}
neither assuming that $f_1,\ldots ,f_m$ have sublinear growth nor that they are linearly independent\footnote{In particular, the number of vector fields $m$ may be arbitrarily large.}, and the time optimal control problem
\begin{equation}\label{controlproblem}
\text{minimize $\{t\geq 0\, | \, y_{x,u}(t)\in \mathcal K \}$ over all  controls }u\in \mathcal U,
\end{equation}
where $y_{x,u}$ denotes the solution of \eqref{eq:s0} associated with $u\in \mathcal U$ such that $y_{x,u}(0)=x\in \R^n$ and $\mathcal K \subset \R^n$ is a nonempety closed set.

We study the regularity of the \textit{minimum time function} 
$$
\mathcal{T}_{\mathcal K}: \R^{n} \longrightarrow [0,+\infty]
$$
under the controllability assumption 
$$\mathcal T_{\mathcal K} (x)<+\infty\quad \forall x\in \R^n.$$
 We are interested in the Lipschitz regularity of the minimum time function. It is well-known that, in general, one cannot expect   $\mathcal T_{\mathcal K}$ to be locally Lipschitz continuous. Indeed, there can be some {\textquotedblleft}singular" time-optimal trajectories implying  the failure of the Lipschitz continuity of $\mathcal T_{\mathcal K}$. In the present context, this can be proved using \cite[Proposition 5.3-(ii)]{BCF} which implies that, along a singular time-optimal trajectory,  the horizontal proximal gradient of $\mathcal T_{\mathcal K}$ contains a  nonzero vector.  Indeed, a necessary condition for the Lipschitz continuity of  $\mathcal T_{\mathcal K}$ at a point $x$ is that the horizontal proximal gradient of $\mathcal T_{\mathcal K}$ at $x$ contains only the vector zero. 
So, without making additional assumptions---such as the absence of singular minimizing controls like in \cite{CR}---the best one can hope for is a sort of ``generic'' local Lipschitz regularity.

In this paper, we show that, if vector fields $f_1,\ldots ,f_m$ are real-analytic and satisfy H\"ormander's bracket generating condition (see \eqref{a:hor} below), then $\mathcal T_{\mathcal K}$ is locally Lipschitz continuous, if $n=2$. Furthermore, if $n=3$, $\mathcal T_{\mathcal K}$ is Lipschitz continuous in the complement of a set of measure zero. 
For $n>3$ the problem remains open. Another interesting open problem is to provide a better Hausdorff estimate for the set of points where $\mathcal T_{\mathcal K}$ fails to be Lipschitz continuous, for $n\geq 3$.

We observe that, since $\mathcal T_{\mathcal K}$ solves---in the viscosity sense---the eikonal equation
\begin{equation}\label{eq:hj}
\sum_{i=1}^m (\nabla v(x)\cdot {f_i(x)} )^2=1,\qquad  x\in \R^n\setminus \mathcal K,
\end{equation}
the above ``generic'' Lipschitz continuity  can be understood as a regularity result for the unique viscosity solution  of (\ref{eq:hj}) (cfr. \cite{B}) under suitable boundary conditions and, possibly, conditions at infinity.

In the case of a control system given by a family of smooth vector fields satisfying H\"ormander's bracket generating condi\-tion, a number of results is present in the mathematical literature assuming either that $\mathcal K$ is a point-wise target or the boundary of an open set. In particular, in \cite[Theorem 11.2]{ABB}, the case of a point-wise target is considered:
the authors show that, for every $y$, the minimum time function $\mathcal T_{\{y\}}$ is smooth on an open dense subset of any compact ball\footnote{The ball is understood with respect to the Carnot-Caratheodory metric defined below.}  centered at $y$. 

%We point out that our result cannot directly be obtained by the previous one. Indeed, even the generic local Lipschitz continuity of
% $\mathcal T_{\mathcal K} =\inf_{y\in \mathcal K} \mathcal T_y$ does not follow directly from the smoothness of $\mathcal T_y$ on an open dense subset %of $\R^n\setminus 
%\{y\}$, since the intersection of open sets needs not be open.

In \cite{N}, for a larger class of control systems and  a compact target satisfying a stronger geometrical assumption than ours, i.e., a uniform interior sphere condition, it is shown that the minimum time function is Lipschitz  on the complement of a set of measure zero.  

In \cite{A1}, a very special model with a smooth target is studied in detail.
In this case, the minimum time function turns out to be  
locally Lipschitz continuous since no singular time-optimal trajectories are present, 
due to the special form of the vector fields (see also \cite{A3} for a geometric explanation of this kind of phenomenon). Again in the case of a smooth target,  
a much more detailed analysis is developed in \cite{ACS1,ACS2}  (see also \cite{A3}). More precisely, in \cite{ACS1} it is shown that a time optimal trajectory is singular if and only if it hits the target at a characteristic point---i.e. a point where all vector fields $f_1,\ldots f_m$ are tangent to the target. Therefore,  $\mathcal T_{\mathcal K}$ is locally Lipschitz continuous and semiconcave in the complement of a set of measure zero. In \cite{ACS2}, it is shown that  $\mathcal T_{\mathcal K}$ is smooth outside a negligible set. 

In the present paper, we follow a similar---yet different---approach: we analyze the structure of the set of all  characteristic points on  the boundary of the \textit{reachable set} from the target at a given time $\tau>0$, i.e.,
\[\mathcal R(\tau )=\{ x\in \R^n\, |\, \mathcal T_{\mathcal K}(x)\le \tau \}.\]
We recall that, when all  data are smooth,   \cite[Th\'eor\`eme 1]{D} ensures that the set of all  characteristic points on a smooth manifold is of measure zero. On the other hand,
such a result is of no use in the present context because  the boundary of $\mathcal R(\tau )$ is nonsmooth. Furthermore, due to the fact that $\mathcal R(\tau )$ is a set with a nonsmooth boundary, its conormal bundle can be only defined using limiting normal cones.

We observe that any characteristic point of $\mathcal R(\tau )$ 
is the endpoint of a singular time-optimal trajectory. Then, using also the fact that we are working in low dimension $(n=2$ or $3)$, the real-analytic regularity of the vector fields imposes severe restrictions 
on the region which can be covered by singular time-optimal trajectories, showing that such a region is empty, if $n=2$, and at most negligible, if $n=3$.  
 
Since Petrov's condition is satisfied  away from characteristic points (see Theorem \ref{prop1}), we can locally reduce our problem to the  regularity of the minimum time function in the presence of a local Petrov condition.
By applying the local regularity result obtained in \cite{S},  
we conclude that the minimum time function is Lipschitz continuous in the complement of a set of measure zero (see Theorem \ref{t:at}). We point out that, since our analysis is local, we do not need restrictions of global nature neither on vector fields $f_1,\ldots ,f_m$ nor on the target.  

Time optimal control problems, in the absence of Petrov's condition, are also considered in \cite{CMN} for differential inclusions which include control systems as a special case. However the set-up of \cite{CMN} does not cover the case of H\"ormander's vector fields of interest to this paper.  

 The outline of this paper is as follows. In Section \ref{s:2}, we recall some basic facts and notations. In Section \ref{s:3}, we study the characteristic points of the boundary of the reachable set. 
 In Section \ref{s:4}, we relate them to the Lipschitz regularity of the minimum time function. In Section \ref{s:5}, we introduce singular time optimal trajectories and we show that 
any characteristic point can be associated with some singular trajectory.

Notice that, in the above sections,  real-analiticity of the vector fields is not assumed.  
On the other hand, in the last section, we show that the assumptions needed for the results of Section \ref{s:4} are automatically verified for real-analytic vector fields in low dimension.

\section{Preliminaries} \label{s:2}

We begin by recalling the definition of the proximal normal cone to a nonsmooth set. We will need this notion in our analysis of the boundary of the reachable set at  given time.  
 Let $C$ be a closed subset of $\R^n$ and $x\in C$. 
The \textit{proximal normal cone}  to $C$ at $x$ is the set defined by 
$$
 	N_{C}^{P}(x)=\{ p\in \R^n\, | \, \exists  \sigma_{x,p}\geq 0 :  p\cdot (y-x)\leq \sigma_{x,p} |y-x|^2,\, \forall y \in C\}.
$$	
\begin{remark}\label{r:i}	
We observe that the proximal normal is a convex and a conic set (i.e. if $p \in N_{C}^{P}(x)$  then $\lambda p \in N_{C}^{P}(x)$, for every  $\lambda \geq 0$).  
Furthermore, we remark that 
\begin{equation}\label{eq:dense} 
N_C^{P}(x)\setminus \{ 0\} \not=\emptyset ,
\end{equation}
for $x$ in a dense subset of $\partial C$. 
\end{remark}

Since we will need a local analysis of the minimum time function with the reachable set at a given time as target, in order to avoid misleading notations, let us consider the minimum time problem with a closed nonempty set $\mathcal R$ as target. We set 
$$
\theta (x,u)=\inf \{ t\geq 0\, | \, y_{x,u}(t)\in \mathcal R\}\in [0,+\infty ],
$$  
where $y_{x,u}$ is a solution of \eqref{eq:s0} . Furthermore, we define 
$$
y(x,u)=y_{x,u}(\theta (x,u)), \quad \text{ provided that }\theta (x,u)<+\infty. 
$$
Given $x\in \R^n\setminus \mathcal R$, we set
$$
{\mathcal F}(x)=\{ \bar x\in\partial {\mathcal R}\, |\, \exists u:\quad \theta (x,u)=T_{\mathcal R}(x)\text{ and }\bar x =y(x,u)\}. 
$$
Then, the following result is a direct consequence of  \cite[Corollary 3.3.]{S}
\begin{theorem}\label{carlo}
Let $f_1,\ldots ,f_m$ be locally $C^{1,1}$ and let $x_0$ be such that $\mathcal T_{\mathcal R}(x_0)<+\infty$. Assume that there exists an open set $A$ containing ${\mathcal F}(x_0)$ such that for every $C\Subset A$ there exists $\mu >0$ such that 
\begin{equation}\label{eq:ipoc}
\min_{u\in U} f(x)u\cdot \eta \le -\mu , 
\end{equation}
for every $x\in C\cap \partial \mathcal R$ and for every $\eta \in N_{\mathcal R}^P(x)\cap \mathbb S^{n-1}$. Then $\mathcal T_{\mathcal R}$ is Lipschitz continuous in a neighborhood of $x_0$. 
\end{theorem}  
Here $\mathbb S^{n-1}\subset \R^n$ is the unit sphere of dimension $n-1$.

In order to state H\"ormander's bracket generating condition, we recall that given two vector fields $f$ and $g$, 	we denote by 
$$
[f,g](x):=Dg(x)(f(x))-Df(x)(g(x))
$$ 
the commutator of $f$ and $g$. Then, given a family of smooth vector fields $f_1,\ldots , f_m$,   we denote by 
$Lie_{\{f_1,\ldots ,f_m\}}$ the Lie algebra generated by the vector fields and by all their iterated commutators. Then, H\"ormander's bracket generating condition can be written as follows 
\begin{equation}\label{a:hor}
Lie_{\{f_1,\ldots ,f_m\}}(x)=\R^n,\qquad \forall x\in \R^n,
 \end{equation} 
where $Lie_{\{f_1,\ldots ,f_m\}}(x)=\{g(x)~|~g\in Lie_{\{f_1,\ldots ,f_m\}}\}$. 
We recall that H\"ormander's bracket generating condition  implies the controllability of the control system \eqref{eq:s0}, i.e. $\mathcal T_{\mathcal K}(x)<+\infty$, for every $x\in \R^n$.

Let us also recall a basic property of  the so-called Carnot-Caratheo\-do\-ry distance\footnote{We will need this property for a localization in the proof of Theorem \ref{t:at}.} (see \cite{NSW} for the definition). 
Assume that the vector fields are smooth and fulfill \eqref{a:hor}. 
 Then, for every $x,y\in \R^n$ the Carnot-Caratheodory distance satisfies 
 $$
 d(x,y)=\mathcal T_{ \{y\}} (x).
 $$ 
 We recall\footnote{See e.g. \cite{NSW}.} that for every compact subset $\mathcal H$ of $\R^n$,  there exist two positive constants $C_1,C_2$ and $\alpha \in ]0,1[$, such that 
 \begin{equation}\label{eq:cc0}
 C_1 |x-y|\le d(x,y)\le C_2 |x-y|^\alpha ,\qquad \forall x,y\in \mathcal H. 
 \end{equation}
 Furthermore, for every $R> 0$ and $x\in \R^n$, we define 
 \begin{equation}\label{eq:cb}
 B^C_R(x)=\{ y\in \R^n\ | \ d (x,y)<R\}.
 \end{equation}
 \begin{remark}\label{r:smco} 
 We observe that, under Condition \eqref{a:hor}, for every $x\in \R^n$ there exists $R_x>0$ such that 
 $$
 \bar B^C_R(x)\text{ is compact for every }R\in ]0,R_x[,  
 $$
 and, by the second inequality in \eqref{eq:cc0}, $\bar B^C_R(x)$ has nonempty interior. 
 (Here $\bar B$ stands for the closure of the set $B$.) 
 \end{remark}

 \section{Characteristic points and Petrov's condition}\label{s:3}

In what follows, we define the Hamiltonian by
$$
H (x,p):=\left (\sum_{i=1}^m f_i(x,p)^2\right )^{\frac 12},\qquad (x,p)\in T^*\R^n,$$
where we have set, with a slight abuse of notation, 
$$
f_i(x,p):=f_i(x)\cdot  p,\quad \forall (x,p)\in T^*\R^n ,\, i=1,\ldots ,m
$$
(note that $T^*\R^n$ can be identified with $\R^n\times \R^n$). 
Observe that 
\begin{equation}\label{eq:p-p}
H (x, -p)=H(x,p),    \quad   \forall (x,p)\in \R^n\times \R^n.
\end{equation}
In order to define the characteristic points, we need the notion of limiting normal cone which can be defined as 
$$
N^L_C(x)=\{ p\in \R^n\ | \ \exists x_h\in C,\quad  x_h\to x,\quad    N_C^{P}(x_h)\ni p_h\to p\}. 
$$
We observe that $N_C^{P}(x)\subset N^L_C(x)$, for every $x\in C$. 
We warn the reader that the following definition is not a standard one.
  \begin{definition}[Characteristic points]
  We say that $x\in \R^n\setminus \mathcal K$ is a characteristic point at time $\tau >0$ if   $ T_{\mathcal K}(x)=\tau$ and there exists $\eta \in N^L_{\mathcal R(\tau )}(x)\setminus\{ 0\}$ such that 
$H(x,\eta )=0$.

In this case, we will write  $x\in E(\tau )$. 
  \end{definition}
We have the following 
\begin{proposition}\label{p:closed}
Let $f_1,\ldots ,f_m$ be locally Lipschitz continuous vector fields. Then, $E(\tau )$ is a closed set, for every $\tau >0$. 
\end{proposition}
We give the proof for the reader's convenience. 
\begin{proof} 
  Let $x_h\in E(\tau )$, $h\in \N$, be a sequence converging to $x$.  We want to show that $x\in E(\tau )$. Since $x_h\in E(\tau )$, we can find $\eta_h\in N^L_{\mathcal R(\tau )}(x_h)\cap \mathbb S^{n-1}$ such that $H(x_h, \eta_h)=0$.
Moreover,  due to the fact that $\eta_h\in \mathbb S^{n-1}$, we may assume that $\eta_h$ converges to $\eta\in \mathbb S^{n-1}$ up to a subsequence still denoted by $\eta_h$. Furthermore, in view of the continuity of $f_1,\ldots ,f_m$, we have that $H(x, \eta )=0$.
We claim that $\eta \in N^L_{\mathcal R(\tau )}(x)$. 
Indeed,  by the definition of $N^L_{\mathcal R(\tau )}(x_h)$, we can find $y_h\in \partial \mathcal R(\tau )$,  $\tilde{\eta}_h\in N^P_{\mathcal R(\tau )}(y_h)$, and a sequence of positive numbers, $a_h\downarrow 0$, such that 
$$
|y_h-x_h|<a_h\quad\text{and}\quad  |\eta_h-\tilde{\eta}_h|<a_h.  
$$
Hence, we conclude that $y_h\to x$ and $\tilde{\eta}_h\to \eta$. This completes the proof.   
\end{proof}

Let us show that, in the complement of $E(\tau )$, the following local Petrov condition holds. 
\begin{theorem}\label{prop1}
Assume that $f_1,\ldots ,f_m$ are locally Lipschitz continuous and let $\tau>0$.  Let us suppose that  
\begin{equation}\label{assumption}
\partial R(\tau ) \setminus E (\tau)\not=\emptyset . 
\end{equation}
	
		Then, for every $y_0\in  \partial R(\tau ) \setminus E(\tau)$ and for every $\delta >0$, with $\bar{B}_\delta (y_0)\cap E(\tau )=\emptyset$, there exists $\mu > 0$ such that  
		\begin{equation}\label{eq:petrov}
			\inf_{u\in \bar{B}^m_1(0)}\; \, \eta\cdot f(y)u  \leq -\mu,
		\end{equation}
		for every $y\in  B_\delta (y_0)\cap \partial R(\tau )$ and for every $\eta \in  N^{P}_{R(\tau)}(y)\cap \mathbb S^{n-1}$.
	\end{theorem}
\begin{proof}
We observe that \eqref{eq:petrov} can be rewritten as 
\begin{equation}\label{eq:petrov2}
H(y, \eta )\geq \mu \qquad  \forall \eta \in  N^{P}_{R(\tau)}(y)\cap \mathbb S^{n-1}. 
\end{equation}
In order to prove this inequality, we argue by contradiction assuming that there exist $y_h\in \overline{B}_\delta (y_0)\cap \partial \mathcal R(\tau )$ and $\eta_h\in N_{\mathcal R(\tau )}^P(y_h)\cap \mathbb S^{n-1}$ so that $y_h\to \bar y\in  \overline{B}_\delta (y_0)\cap \partial \mathcal R(\tau )$,  $\eta_h \to \bar\eta \in \mathbb S^{n-1}$, and $H(\bar y , \bar \eta )=0$.
 Hence, we find that  $\bar \eta \in N^L_{\mathcal R(\tau )}(y)$ and $\bar y\in E(\tau )\cap \overline{B}_\delta (y_0)$, in contrast with the fact that $\overline{B}_\delta (y_0)\cap E (\tau )=\emptyset$. Then, \eqref{eq:petrov2} follows and the proof is complete. 
\end{proof}

\section{Regularity of the minimum time function}\label{s:4}

In this section we give an "abstract" result, the assumptions of which may seem hard to check. 
However,  in the next section, we will show that such assumptions are automatically fulfilled for real-analytic vector fields in low dimension.

We begin with the following 
 
\begin{definition}
For every real-valued function $g$, defined on an open subset $\Omega \subset \R^n$, the \textit{Lipschitz singular support} of $g$ is the smallest closed subset of $\Omega$ such that $g$ is locally Lipschitz outside such a set.  We will denote the Lipschitz singular support of $g$ by $\sing\supp_{Lip}g$.
\end{definition} 
In other words if $x\notin \sing\supp_{Lip}g$ then there exists an open set containing $x$, $V_x$, such that $g$ is Lipschitz continuous on $V_x$. 

\begin{theorem}\label{t:at}
  Let $\mathcal K\subset \R^n$ be a nonempty closed set, assume that $f_1,\ldots ,f_m$  are  of class $C^{\infty}$ and satisfy H\"ormander's condition \eqref{a:hor}\footnote{We point out that the above assumptions are stated in a global form for the sake of simplicity. Indeed, since Theorem \ref{t:at} is a local result, it suffices to assume them fulfilled in a neighborhood of a point of interest.}.
Furthermore, suppose that, given $x_0\in \R^n\setminus \mathcal K$, there exist $\epsilon >0$ and a closed set $N$ of measure zero such that 
\begin{equation}\label{eq:loccp}
B_\epsilon (x_0)\cap E(\tau )\subset N, \qquad \forall \tau >0.
\end{equation}
Then,  $\sing\supp_{Lip} T_{\mathcal K}\cap B_\epsilon (x_0)\subset N$. In particular,   $\sing\supp_{Lip} T_{\mathcal K}\cap B_\epsilon (x_0)$ is a zero measure set. 
\end{theorem}
\begin{proof}
Let $x_0\in \R^n\setminus \mathcal K$, let $\epsilon>0$ and let $N$ be a closed set of measure zero such that
  \begin{equation}\label{eq:assu1}
B_\epsilon (x_0)\cap E(\tau )\subset N, \qquad \forall \tau >0. 
  \end{equation}
We want to show that, for every $x\in  B_\epsilon (x_0)\setminus N$, there exists $\delta >0$ such that 
 $T_{\mathcal K}$ is Lipschitz continuous on $B_\delta (x)$. 

Let  $x\in  B_\epsilon (x_0)\setminus N$, then we can find three numbers $\delta_1>\delta_2>\delta_3>0$ so that
\begin{itemize}
\item[$(A)$] $\bar{B}^C_{\delta_1}(x)$ is a compact subset of $B_\epsilon (x_0)$ and $B^C_{\delta_1}(x)\cap N=\emptyset$;
\item[$(B)$]  $\delta_1>3\delta_2$. 
\end{itemize}
 In particular, we notice that
 $$
B^C_{\delta_3}(x)\subset B^C_{\delta_2}(x)\subset   B^C_{3\delta_2}(x)\subset B^C_{\delta_1}(x).
 $$
 Take $\tau >0$ such that
 \begin{equation}\label{eq:deftau}
\tau <\min \{ T_{\mathcal K}(z)\, | \,  z\in B^C_{\delta_3} (x)\}\quad \text{ and }\quad  {\mathcal R}(\tau )\cap B^C_{\delta_2} (x)\not=\emptyset .   
 \end{equation}
 We claim that for every $y\in  B^C_{\delta_3} (x)$, for every $z_1\in \partial \mathcal R(\tau )\cap B^C_{\delta_2}(x)$ and for every $z_2\in \partial \mathcal R(\tau )\setminus B^C_{\delta_1}(x)$ we have that
 \begin{equation}\label{claim}
d(y,z_2)>d(y,z_1). 
   \end{equation}
 Indeed, by the triangular inequality
 $$
d(y,z_2)\geq d(x,z_2)-d(x,y)>\delta_1 -\delta_3,   
 $$
and
$$
d(y,z_1)\le d(y,x)+d(x,z_1)<\delta_3+\delta_2 . 
$$
Hence, the above inequalities and condition $(B)$ yield \eqref{claim}. 
Let $\delta >0$ be such that $B_\delta (x)\subset B^C_{\delta_3}(x)$.

Let us observe that, for every $y\in B_\delta (x)$, 
\begin{equation}\label{eq:datat}
\mathcal T_{\mathcal K} (y)=\tau +\mathcal T_{\mathcal R(\tau )}(y).
\end{equation}  
(We recall that the symbol $\mathcal T_{\mathcal R(\tau )}$ stands for the minimum time function associated with equation \eqref{eq:s0} and target $\mathcal R(\tau)$.) Furthermore, for every $y\in B_\delta (x)$, we have that 
$$
\mathcal T_{\mathcal R(\tau )}(y)=\inf_{z\in \partial \mathcal R(\tau ) } d(y,z).
$$
In view of \eqref{claim}, the infimum above is a minimum and is attained on $\partial \mathcal R(\tau )\cap \bar B^C_{\delta_2}(x)$  (which, by construction, is disjoint from the set $E(\tau)$). Then, by Theorem \ref{prop1}, the assumptions of Theorem \ref{carlo} are satisfied and, possibly reducing $\delta >0$,  $\mathcal T_{R(\tau )}$ is Lipschitz continuous on $B_\delta (x)$. Hence, due to \eqref{eq:datat}, $\mathcal T_{\mathcal K}$ is Lipschitz on $B_\delta (x)$. This completes the proof. 
\end{proof} 
\begin{remark}\label{r:ponte}
(i) We observe that, if in \eqref{eq:loccp} one can take $N=\emptyset$, then Theorem \ref{t:at} implies the Lipschitz continuity of $\mathcal T_{\mathcal K}$ on $B_\epsilon (x_0)$. 

\noindent 
(ii) If  \eqref{eq:loccp}  holds at every point in $\R^n\setminus \mathcal K$, then, 
using a compact exhaustion and the fact that the union of countably many sets of measure zero is a set of measure zero, one has that $\sing\supp_{Lip}\, \mathcal T_{\mathcal K}$ is a set of measure zero. 
\end{remark} 
We observe that, as a consequence of Theorem \ref{t:at} and Rademacher's Theorem, we have  the following 
\begin{corollary}\label{c:diff}
If the assumptions of Theorem \ref{t:at} are satisfied, then $\mathcal T_{\mathcal K}$ is differentiable a.e. on $B_\epsilon (x_0)$. 
\end{corollary} 

\section{Singular time-optimal trajectories}\label{s:5}

In this section we discuss the relationship between 
characteristic points and some special time-optimal trajectories. We begin with the following  
 \begin{definition}[Singular time optimal trajectories]
 For $x\notin \mathcal K$, $y_{x,u}$ is a singular time-optimal trajectory, if   $y_{x,u}( T_{\mathcal K}(x))\in \mathcal K$ and there exists a Lipschitz arc $p :[0,  T_{\mathcal K}(x)]\rightarrow \R^n\setminus \{0\}$ such that
 \begin{equation}\label{eq:sot}
			\begin{cases}
p'(t)=-((\nabla_x f)(y_{x,u}(t)) u(t))^\star  p(t),\, \text{for a.e.   }t\in [0,T_{\mathcal K}(x)],   
\\
p(0)\in -N_{R(\tau )}(x) ,		 
			\end{cases}
		\end{equation}
and		
\begin{equation}\label{eq:sot22}
H(y_{x,u}(t),p(t)) =0\quad \text{for every   }t\in [0,T_{\mathcal K}(x)].   
\end{equation} 
		  
\end{definition}

In order to give necessary and sufficient conditions for a point to be a characteristic point in terms of singular time-optimal trajectories, we need some preliminary results. 
We consider  the minimum time function with target $\overline{\R^n\setminus \mathcal R(\tau )}$. We need the following version of the Maximum Principle, see e.g. \cite[Theorem 3.1]{CFS}.
  \begin{theorem}\label{t:suff0}
Let $f_1,\ldots ,f_m$ be locally $C^{1,1}$ and  let $\tau >0$. Let us fix arbitrarily $x \in \partial \mathcal R(\tau)$,  let  $\bar x\in \mathcal K$ be such that $y=y_{\bar x,u}$ is 
 a time-optimal trajectory with target $\overline {\R^n\setminus \mathcal R(\tau )}$. Then, for every 
 $\eta \in N^P_{\mathcal R(\tau )}(x)\cap \mathbb S^{n-1}$, there exists a Lipschitz arc, $p :[0,\tau]\rightarrow \R^n\setminus \{0\}$, such that the pair $(y,p)$ satisfies 
$$
\begin{cases}
y'(t)=f(y(t))u(t) , 
\\
p'(t)=-((\nabla_x  f)(y (t))u(t))^* p(t),  
\end{cases}
$$
 for a.e. $t\in [0,\tau ]$, with 
 $$
 y(0)=\bar x, \quad y(\tau )=x\text{ and }p(\tau )=-\eta ,
 $$
 and
 $$
 -f(y(t))u(t)\cdot p(t)=  H(y(t),p(t)), \qquad \text{for a.e. } t\in [0,\tau ].  
 $$
 \end{theorem}
Observe that,  if $p(\cdot )$ satisfies the above equations, then $\lambda p(\cdot )$ satisfies the same equations for any $\lambda >0$. Thus, inverting the direction of the time variable, 
Theorem \ref{t:suff0} can be rewritten in the following useful form. 
  \begin{theorem}\label{t:suff0bis}
Let $f_1,\ldots ,f_m$ be locally $C^{1,1}$ and  let $\tau >0$. Let us fix arbitrarily $x \in \partial \mathcal R(\tau)$,  let  $y=y_{x,u}$ be 
 a time-optimal trajectory with target $\mathcal K$ an let $\bar x:=y(\tau ) \in \mathcal K$. Then, for every 
 $\eta \in N^P_{\mathcal R(\tau )}(x)\setminus \{ 0\}$, there exists a Lipschitz arc, $p :[0,\tau]\rightarrow \R^n\setminus \{0\}$, such that
\begin{equation}\label{eq:bh0bis}
\begin{cases}
y'(t)=f(y(t))u(t) , 
\\
p'(t)=-((\nabla_x  f)(y (t))u(t))^* p(t),  
\end{cases}
\end{equation} 
 for a.e. $t\in [0,\tau ]$, with 
 \begin{equation}\label{eq:bc0bis}
 y(0)=x, \quad y(\tau )=\bar x\text{ and }p(0)=-\eta ,
 \end{equation}
 and
 \begin{equation}\label{eq:maxh0bis}
 f(y(t))u(t)\cdot p(t)= H(y(t),p(t)),   \qquad \text{for a.e. } t\in [0,\tau ].  
 \end{equation}
 \end{theorem}
We will need also the following 
\begin{lemma}\label{l:hcost}
Under the assumptions of Theorem \ref{t:suff0bis}, we have that 
\begin{equation}\label{eq:hcost}
 H(y(t),p(t))=H(x,\eta ),\qquad \text{ for every }t\in [0,\tau ].
\end{equation}
\end{lemma} 
\begin{proof} 
Set 
$$
[0,\tau ]\ni t\mapsto g(t):= H( y(t),  p(t)).
$$
We observe that, as $g$ is Lipschitz (this follows from the Lipschitz continuity of   $y(\cdot )$ and $p (\cdot )$, and the regularity of $H(\cdot ,\cdot )$), the proof reduces to show that 
$g'(t)=0$ for a.e. $t\in [0,\tau ]$.   
Since $ p (t )$ satisfies the adjoint equation \eqref{eq:bh0bis}, \ref{eq:maxh0bis} and   
$$
H(x,p)=\max_{u\in \bar{B}_1(0)} (f(x)u\cdot p ) , 
$$
we have that 

	\begin{eqnarray*}
			g'(t)&=& \lim_{\delta\to 0^+}\frac{ g(t+\delta)-g(t)}{\delta}
			\\
			&\geq &
			\lim_{\delta \to 0^+}\frac{p(t+\delta)\cdot f(y(t+\delta)) u(t) -\,  p(t)\cdot f(y(t))u(t) }{\delta}
			\\
			&=&\lim_{\delta \to 0^+}\frac{\,(p(t+\delta)-p(t))\cdot f( y(t+\delta))u(t) }{\delta}
			\\
			&&\qquad  +\frac{p(t)\cdot( f( y(t+\delta)) u(t)-f( y(t)) u(t)) }{\delta}
			\\
			&=&  p'(t)\cdot f( y(t))u(t) +p(t)\cdot (\nabla_xf)( y(t))  u(t)f(y(t)) u(t)
			\\
			&=& 0.
	\end{eqnarray*}
Similarly, we have that 
		\begin{eqnarray*}
			g'(t)&=& \lim_{\delta\to 0^-}\frac{ g(t+\delta)-g(t)}{\delta}
			\\
			&\leq &
			\lim_{\delta \to 0^-}\frac{p(t+\delta)\cdot f(y(t+\delta)) u(t) -\,  p(t)\cdot f(y(t))u(t) }{\delta}
			\\
			&=&\lim_{\delta \to 0^-}\frac{\,(p(t+\delta)-p(t))\cdot f( y(t+\delta))u(t) }{\delta}
			\\
			&&\qquad   +\frac{p(t)\cdot( f( y(t+\delta)) u(t)-f( y(t)) u(t)) }{\delta}
			\\
			&=&  p'(t)\cdot f( y(t))u(t) +p(t)\cdot (\nabla_xf)( y(t))  u(t)f(y(t)) u(t)
			\\
			&=& 0.
	\end{eqnarray*}
Then, we find that if $g$ is differentiable at $t$, then $g'(t)=0$. 
This completes the proof.  
\end{proof}

 \begin{theorem}\label{l:cp}
   Assume that  $f_1,\ldots ,f_m$ are locally of class $C^{1,1}$. Then, the following assertions are equivalent
   \begin{enumerate}
   \item $x\in E(\tau )$ for some $\tau >0$;
   \item there exists a singular time-optimal trajectory, $y_{x,u}$.
     \end{enumerate}
 \end{theorem}
 \begin{proof}
The implication  $(2)\implies (1)$ follows from the definition of characteristic point and \eqref{eq:sot22}. 

Let us prove that  $(1)\implies (2)$.   
We need to distinguish two cases: 

\noindent 
(i) $H(x,\eta )=0$ for some $\eta\in N^P_{\mathcal R(\tau )}(x)\setminus \{0\}$;

\noindent 
(ii) $H(x,\eta )=0$ for some $\eta\in N^L_{\mathcal R(\tau )}(x)\setminus N^P_{\mathcal R(\tau )}(x)$, i.e. there exists $x_h\in \partial \mathcal R(\tau )\setminus\{ x\}$ and $\eta_h \in N^P_{ \mathcal R(\tau )}(x_h)\setminus \{ 0\}$ such that $x_h\to x$ and $\eta_h \to \eta \not= 0$ with $H(x,\eta )=0$.

We begin with (i). 
If $\tau >0$ and $x\in \partial \mathcal R(\tau )$ then there exists a control $u$ such that $\bar x:=y_{x,u}(\tau )\in \mathcal K$. 
Then, in view of Theorem~\ref{t:suff0bis} , there exists a Lipschitz arc $ p :[0,  \tau ]\rightarrow \R^n\setminus \{0\}$ such that
 \begin{equation}\label{eq:sot2}
			\begin{cases}
			y'(t)=f(y(t))  u(t),\hspace{1.7cm}  y (0)=x,\quad y(\tau )=\bar x,
			\\ 
		p'(t)=-((\nabla_x f)  ( y(t)) u(t))^{\star}  p(t),		\quad  p(0 )=-\eta ,  
			\end{cases}
		\end{equation}
		 for a.e.   $t\in [0,\tau ]$. Furthermore, we have that 
	$$
 p(t)\cdot f( y(t)) u(t) =H( y (t), p(t))\quad \text{ for a.e. }t\in[0,\tau].
	$$
	In order to conclude that $y_{x,u}$ is a singular time-optimal trajectory it remains to show that 

	\begin{equation}\label{eq:h0}
	H(y(t),  p(t))=0,\qquad \text{ for every  }t\in [0,\tau]. 
	\end{equation}
\eqref{eq:h0}  is a direct consequence of Equation \eqref{eq:hcost} and the fact that $H(x,\eta )=0$. 
This completes our proof in case (i). 

Now, let us consider case (ii).    We want to show that  there exists a singular time-optimal trajectory $y_{x,u}$.  Let $u_h(\cdot )$ be such that $\bar x_h:=y_{x_h,u_h}(\tau )\in \mathcal K$.  
Using once more Theorem~\ref{t:suff0bis} we find that, for every $h\in \N$, there exists a sequence of Lipschitz arcs $p_h :[0,  \tau ]\rightarrow \R^n\setminus \{0\}$ such that
\begin{equation}\label{eq:prelim}
			\begin{cases}
			y'_h(t)=f( y_h(t)) u_h(t) ,\hspace{1.7cm}  y_h(0)= x_h,\quad  y_h(\tau )=\bar x_h 
			\\ 
		 p_h'(t)=-    ((\nabla_{x} f)(y_{h}(t)) u_{h}(t))^{\star}  p_h(t),		\quad  p_h(0)=-\eta_h ,  
			\end{cases}
\end{equation}
		 for a.e.   $t\in [0,\tau ]$.   In addition, we have that
\begin{equation}\label{eq:prelimbis}
p_h(t)\cdot f( y_h(t))u_h(t) =H(y_h(t),p_h(t))\quad \text{ for a.e. }t\in[0,\tau].
\end{equation} 

Since $ p_h(0 )$ is in a compact set,   then,   by the Ascoli-Arzel\`a Theorem, 
$(y_{x_h,u_h} (\cdot ),  p_h(\cdot ))$ uniformly converges up to extracting a subsequence. We observe that the sequence of controls, $u_h:[0,\tau ]\longrightarrow \bar B^m_1(0)$ is bounded in $L^2$. Then, $u_h$ converges, weakly in $L^2$, to $ u: [0,\tau ]\longrightarrow \R^m$ up to extracting a subsequence which we still denote by $ u_h$. Now, using the Banach-Saks theorem, we deduce that there exists a subsequence of $ u_h$, $ u_{h_j}$, such that  $(u_{h_1}+\ldots + u_{h_j})/j$, strongly converges  in $L^2$ to $ u$. Then, $( u_{h_1}+\ldots + u_{h_j})/j$  converges  to $u$, for a.e. $t\in [0,\tau ]$, and we deduce that $ u:[0,\tau ]\longrightarrow \bar B^m_1(0)$.
  
Furthermore, rewriting \eqref{eq:prelim} in integral form 
\begin{equation}\label{eq:intform}
	\begin{cases}
			 y_h(t)=x_h+\int_0^t f( y_h(s)) u_h(s) \, ds,
			\\ 
		 p_h(t)=-\eta_h -\int_0^t    ((\nabla_{x} f)( y_{h}(s))u_{h}(s))^{\star}  p_h(s) \, ds,  
			\end{cases}
\end{equation}
taking the limit as $h\to \infty$ in \eqref{eq:intform}, and recalling that  $( y_h(\cdot ), p_h (\cdot ))$ uniformly converges to  $( y(\cdot ),p (\cdot ))$ while $ u_h (\cdot )$ weakly converges in $L^2$ to $ u (\cdot )$, we conclude that    
\begin{equation}\label{eq:yp}
\begin{cases}
			 y(t)= x+\int_0^t f( y(s))u(s) \, ds,
			\\ 
		 p(t)=-\eta +\int_0^t    ((\nabla_{x} f)( y (s)) u(s))^{\star}  p(s) \, ds. 
			\end{cases}
\end{equation}
We observe that, since the uniform limit of time-optimal trajectories is in turn time-optimal (see e.g. \cite{LM}),  $y(\cdot )$  is a time-optimal trajectory. 
It order to complete the proof of the lemma, it remains to show that   $ y(\cdot )$ is  singular, i.e.  
$$
p(t)\cdot f( y(t)) u(t) =H(y(t),p(t))=0\quad \text{ for every }t\in[0,\tau].
$$
Integrating \eqref{eq:prelimbis} between $0$ and $\tau$ and taking the limit as $h\to \infty$, we deduce that 
\begin{equation}\label{5}
\int_0^\tau  p(t)\cdot  f( y(t))u(t) \, dt =\int_0^\tau H( y(t),p(t))\, dt.  
\end{equation}
Since 

$$
H(y(t),p(t))-p(t)\cdot f(y(t))u(t)\geq 0, \quad \text{ or a.e. }t\in [0,\tau ],  
$$ 
\eqref{5} yields   

	$$
  p(t)\cdot f( y(t))u(t) =H( y(t), p(t))\quad \text{ for a.e. }t\in[0,\tau].
	$$
	In order to complete the proof, it remains to show that
	\begin{equation}\label{217}
	H( y(t), p(t))=0,\qquad \text{ for every  }t\in [0,\tau].
	\end{equation}
Once more, due to Equation \eqref{eq:hcost} and $H(y(0), p(0))=0$,  \eqref{217} follows. 
This completes our proof of the theorem. 
\end{proof}
\begin{remark}
We observe that $E(\tau )=\emptyset$ for every $\tau >0$ if and only if the minimum time problems admits no singular time-optimal trajectories. 
 \end{remark} 
\section{Applications}\label{s:6} 

In order to discuss some applications of our ``abstract'' regularity result, let 
us define the characteristic set associated with vector fields $f_1,\ldots ,f_m$ as 
\begin{multline*}
  \Char (f_1,\ldots ,f_m)
\\
  =\{ (x,p )\in \R^n\times (\R^n\setminus \{0\} )\, |\, f_i(x,p)=0,\, \forall i=1,\ldots ,m\}.
\end{multline*}
Let 
$$
\pi :\R^n\times (\R^n\setminus \{0\} )\longrightarrow \R^n
$$
be the projection on the base $\pi (x,p)=x$.

We recall that a manifold $\Sigma\subset \R^{2n}$ is symplectic  if the restriction of the symplectic form $\sigma =\sum_{j=1}^n dp_j\wedge dx_j$ to  $\Sigma$ is nondegenerate, or equivalently $T_\rho \Sigma \cap T_\rho \Sigma ^\sigma =\{ 0\}$, for every $\rho \in\Sigma$. (Here the exponent $\sigma$ stands for the orthogonal with respect to the symplectic form.)

We need the following 
\begin{lemma}\label{l:lemma}
  Let $\mathcal K\subset \R^n$ be a nonempty closed set, assume that vector fields, $f_1,\ldots ,f_m$,  are locally of class $C^{1,1}$, and let $V\subset 
\R^n\setminus \mathcal K$ be an open set. 
If either $\pi^{-1}(V)\cap \Char (f_1,\ldots ,f_m)=\emptyset $ or $\pi^{-1}(V)\cap \Char (f_1,\ldots ,f_m)$ is a symplectic manifold, then $E(\tau )\cap V=\emptyset$ for every $\tau >0$.     
\end{lemma} 
\begin{proof}
The implication 
$$
\pi^{-1}(V)\cap \Char (f_1,\ldots ,f_m)=\emptyset \implies \forall \tau >0:\, E(\tau )\cap V=\emptyset
$$
 trivially follows from the definition of a singular time-optimal trajectory. The second implication is well known,  we give a sketch of its proof for the reader's convenience.  Denoting by $\rho (\cdot )=(y(\cdot ),p(\cdot ))$ a pair consisting of a singular time-optimal trajectory and the corresponding solution $p$ of \eqref{eq:sot}.
We recall that $\rho (\cdot )$ satisfies the following broken Hamiltonian system
\begin{equation}\label{12}
\rho '(t)=\sum_{j=1}^m u_j(t)H_{f_j}(\rho (t)), \qquad \text{ for a.e. }t\in I, 
\end{equation}
for a suitable interval $I$. (Here $H_{f_j}$ 
is the unique Hamiltonian vector field associated with the symbol $f_j(\rho )=f_j(x,p)$, i.e. $df_j (\rho )v = \sigma (v, H_{f_j} (\rho ))$, for every $v \in T_\rho \Char (f_1,\ldots ,f_m)$). We have that, whenever it exists, $\rho '(t )$ is a tangent vector to $\Char (f_1,\ldots ,f_m)$ at $\rho (t)$, i.e. 
\begin{equation}\label{13}
\rho '(t)\in T_{\rho (t)} \Char (f_1,\ldots ,f_m).
\end{equation}
We observe that 
$$
df_i(\rho)(v)=\sigma (v, H_{f_i}(\rho))=0, \qquad \forall v\in  T_{\rho} \Char (f_1,\ldots ,f_m),
$$ 
and for every $i=1,\ldots ,m$. Then, we find
$$
\Span \{ H_{f_1}(\rho ),\ldots , H_{f_m} (\rho)\}\subset T_\rho \Char (f_1,\ldots ,f_m)^\sigma .
$$
Hence, due to \eqref{12}, we have that 
$$
\rho '(t)\in T_{\rho (t)} \Char (f_1,\ldots ,f_m)^\sigma ,
$$
and, in view of \eqref{13}, we conclude that 
$$
\rho '(t)\in T_{\rho (t)} \Char (f_1,\ldots ,f_m)\cap T_{\rho (t)} \Char (f_1,\ldots ,f_m)^\sigma .
$$
On the other hand, since $\Char (f_1,\ldots ,f_m)$ is a symplectic manifold, we find that $\rho '(t)=0$ for a.e. $t$ in a suitable connected interval $I$ and $I\ni t\mapsto \rho (t)=(y(t),p(t))$ is a constant curve (in contrast with the assumption that $y(\cdot )$ is a time-optimal trajectory).  
  \end{proof} 

As a direct consequence of Theorem \ref{t:at} and Lemma \ref{l:lemma}, we have the following 
\begin{theorem}\label{t:app1}
  Let $\mathcal K\subset \R^n$ be a nonempty closed set,
assume that the vector fields, $f_1,\ldots ,f_m$,  are  of class $C^{\infty}$ and satisfy H\"ormander condition \eqref{a:hor}.
  Suppose that $\Char (f_1,\ldots ,f_m)$ is a symplectic manifold.
  Then, $\sing\supp_{Lip}\, T_{\mathcal K}=\emptyset$. 
\end{theorem}
 
We observe that the result is known if target $\mathcal K$ is the complement of a bounded open domain $\Omega\subset \R^n$ with smooth boundary, that is, $T_{\mathcal K}$  is the so called exit time from $\Omega$ (see \cite{ACS1}).

\subsection{The analysis of the real analytic-case in low dimension}  
In this section, we make the following assumption 
\vspace{0,5cm} 

\noindent {\bf (H)} vector fields $f_1,\ldots , f_m$ are real-analytic and satisfy H\"or\-man\-der's bracket generating condition, \eqref{a:hor}. 

\vspace{0,5cm}
 
We begin with the case of dimension $2$. We need a preliminary result the proof of which can be found in  \cite[Theorem 1.1 (i)]{AM} (see also \cite[Chapter 3]{AB}).
\begin{lemma}\label{l:d2} 
Let $n=2$, assume {\bf (H)}, and let $\mathcal K$ be a nonempty closed set. Then, for every $x\in \R^2\setminus \mathcal K$ there exist an open neighborhood of $x$, $W$, and $w_1,\ldots ,w_\ell \in W$ so that for every $y\in W\cap \pi (\Char (f_1,\ldots ,f_m))\setminus  \{ w_1,\ldots ,w_\ell \}$ there exists $\epsilon >0$ such that 
\begin{equation} 
\Char (f_1,\ldots ,f_m)\cap \pi^{-1}(B_\epsilon (y) ) \ \text{ is a symplectic manifold.}
\end{equation} 
\end{lemma} 
\begin{theorem}\label{t:rac2}
Let $n=2$, assume {\bf (H)}, and let $\mathcal K$ be a nonempty closed set. Then 
$T_{\mathcal K}$ is locally Lipschitz continuous on $\R^2\setminus \mathcal K$. 
\end{theorem}  
\begin{proof}
The result is a consequence of Theorem \ref{t:at}. We only need to show that $E(\tau )=\emptyset$, for every $\tau >0$. 
Let $x\in \R^2\setminus \mathcal K$ and let $W$ be the  open neighborhood of $x$ given by Lemma \ref{l:d2}. 
The idea of the proof consists of showing that, due to the fact that we are working in dimension $2$, the support of a singular time-optimal trajectory is, at most, a discrete subset of $W$, i.e., singular time-optimal trajectories do not exist.

Let  $y(t)\in W$ be a singular time-optimal trajectory, for every $t\in ]a,b[$ (for some $a<b$).
Hence, in view of Lemma \ref{l:lemma}, we find that 
 $y(t)\in \{   w_{i}\ | \ , i=1,\ldots ,\ell \}$ for every $t\in ]a,b[$ (with $w_{i}$ as in Lemma \ref{l:d2}).  
It follows that $E(\tau )\cap W=\emptyset$, for every $\tau >0$, and this completes our proof. 
\end{proof}
\begin{remark}
We point out that the above Lipschitz regularity result is strictly related to the driftless structure of control system \eqref{eq:f}. For example, for the double integrator $\ddot{ x}=u\in [-1,1]$, with point target $(0,0)$ it is well-known that the minimum time function fails to be Lipschitz continuous (it is H\"older continuous of exponent $1/2$). 
\end{remark} 
In the case of three variables, we need the following preliminary result (see \cite[Theorem 1.1 (ii)]{AM}). 
\begin{lemma}\label{l:3d} 
Let $n=3$, assume {\bf (H)} and let $\mathcal K$ be a nonempty closed set. Then, for every $x\in \R^3\setminus \mathcal K$ there exist an open neighborhood of $x$, $W$, and a real-analytic set of dimension $2$, $\mathcal{A}\subset W$, so that 
\begin{eqnarray*} 
& &\Char (f_1,\ldots ,f_m)\cap \pi^{-1}(W\setminus \mathcal{A} ) \ \text{either is a symplectic manifold,} 
\\
& &\text{or it is the empty set.}
\end{eqnarray*} 
\end{lemma} 
As a consequence of the above result we have 
\begin{theorem}\label{t:rac3}
Let $n=3$, assume {\bf (H)} and let $\mathcal K$ be a nonempty closed set. Then 
$\sing\supp_{Lip}\, T_{\mathcal K}$ is a set of measure zero.
\end{theorem} 
\begin{remark}\label{r:opt}
We observe that, without additional assumptions, it may happen that $\sing\supp_{Lip}\, T_{\mathcal K}\not=\emptyset$. Indeed, in \cite[page 3326]{ACS1}, it is given an open bounded set with smooth boundary $\Omega\subset \R^3$ such that taking
  $$
\mathcal{K}=\R^3\setminus \Omega ,\quad  f_1=\partial_{x_1},\text{ and } f_2=(1-x_1)\partial_{x_2} +x_1^2\partial_{x_3},  
  $$
$\mathcal T_{\mathcal K}$ is not locally Lipschitz continuous.  
\end{remark} 
\begin{proof}
Once more, the idea of the proof consists of applying Theorem \ref{t:at} to the present case. 
For this purpose, let $x\in \R^3\setminus \mathcal K$. We want to show that one can find a neighborhood of $x$, $W$, such that 
$$
E(\tau )\cap W \subset N 
$$
for a suitable closed set $N$ of measure zero and for every $\tau >0$. 
Let $W$ be the set given by Lemma \ref{l:3d}.  
Using once more Lemma \ref{l:lemma} and Lemma \ref{l:3d}, we deduce that, if $y(\cdot )$ is a singular time-optimal trajectory and $y(t)\in W$, for every $t$ in a suitable interval $]a,b[$, then   $y(t)\in \mathcal A $, for every $t\in ]a,b[$. Hence,  for every $\tau>0$, 
$$ 
E(\tau )\cap W\subset \mathcal A . 
$$ 
Then, possibly reducing $W$, we conclude that there exists $N$, a closed set of measure zero, such that 
$$
E(\tau )\cap W\subset N,
$$
for every $\tau >0$. 
 This completes our proof. 
  \end{proof}

\section*{Declarations}
\begin{itemize}
\item {\bf Conflict of interest:} On behalf of all authors, the corresponding author states that there is no conflict of interest. 
\item{\bf Funding:} This work was partly supported by the National Group for Mathematical Analysis, Probability and Applications (GNAMPA) of the Italian Istituto Nazionale di Alta Matematica ``Francesco Severi''; moreover, the third author acknowledges support by the Excellence Department Project awarded to the Department of Mathematics, University of Rome Tor Vergata, CUP E83C18000100006. 
\end{itemize}


\begin{thebibliography}{}

 \bibitem{ABB} {\sc A. Agrachev, D. Barilari and U. Boscain}, A comprehensive introduction to sub-Riemannian geometry. From the Hamiltonian viewpoint.   Cambridge Studies in Advanced Mathematics, 181. Cambridge University Press, Cambridge, 2020.
 
 \bibitem{A1} {\sc P. Albano}, {\it On the eikonal equation for degenerate elliptic operators}, Proc. Amer. Math. Soc. 140 (2012), no. 5, 1739--1747.
  
\bibitem{A3} {\sc P. Albano}, Some remarks on the Dirichlet problem for the degenerate eikonal equation. Trends in control theory and partial differential equations, 1--16, Springer INdAM Ser., 32, Springer, Cham, 2019.

\bibitem{AB} {\sc P. Albano and A. Bove}, Wave front set of solutions to sums of squares of vector fields. Mem. Amer. Math. Soc. 221 (2013), no. 1039.
  
\bibitem{ACS1} {\sc P. Albano, P. Cannarsa and T. Scarinci}, {\it Regularity results for the minimum time function with H\"ormander vector fields}, J. Differential Equations 264 (2018), no. 5, 3312--3335.

\bibitem{ACS2}  {\sc P. Albano, P. Cannarsa and T. Scarinci}, {\it Partial regularity for solutions to subelliptic eikonal equations}, C. R. Math. Acad. Sci. Paris 356 (2018), no. 2, 172--176.

\bibitem{AM} {\sc P. Albano and M. Mughetti}, { \it  On the analytic singular support for the solutions of a class of degenerate elliptic operators},  Pure Appl. Anal. 3 (2021), no. 3, 473--486.


\bibitem{B} {\sc M. Bardi}, {\it A boundary value problem for the minimu-time function}, SIAM J. Control Optim. 27 (1989), no. 4, 776--785.

\bibitem{BCF} {\sc V. Basco, P. Cannarsa, and H. Frankowska}, {\it Semiconcavity results and sensitivity relations for the sub-Riemannian distance}, Nonlinear Analysis (2019),  no. 184, 298--320.

\bibitem{BT} {\sc A. Bove and F. Treves}, {\it On the Gevrey hypo-ellipticity of sums of squares of vector fields}, 
Ann. Inst. Fourier (Grenoble) {\bf (54)} (2004), no. 5, 1443--1475. 
   % 
  
\bibitem{CFS}{\sc P. Cannarsa, H. Frankowska and C .Sinestrari}, {\it Optimality conditions and synthesis for the minimum time problem},
Set-valued analysis in control theory. Set-Valued Anal. 8 (2000), no. 1-2, 127--148.

\bibitem{CMN} {\sc P. Cannarsa, A. Marigonda and K.T. Nguyen}, {\it Optimality conditions and regularity results for time optimal control problems with differential inclusions}, 
J. Math. Anal. Appl. 427 (2015), no. 1, 202--228. 

 
  \bibitem{CR}{\sc P. Cannarsa and L. Rifford}, {\it Semiconcavity results for optimal control problems admitting no singular minimizing controls}, Ann. Inst. H. Poincar\'e Anal. Non Lin\'eaire 25 (2008), 773--802. 
 
%
\bibitem{D} {\sc M. Derridj}, {\it Sur un th\'eor\`eme de traces}, Ann. Inst. Fourier (Grenoble) 22 (1972), 73-–83. 

\bibitem{LM} {\sc E.B. Lee and L. Markus}, Foundations of optimal control theory. John Wiley \& Sons, Inc., New York-London-Sydney 1967.

\bibitem{L}{\sc S.Lojasiewicz}, Ensembles semi--analytiques, IHES 1967.
  

\bibitem{NSW} {\sc A. Nagel, E.M. Stein and S. Wainger}, {\it Balls and metrics defined by vector fields. I. Basic properties} Acta Math. 155 (1985), no. 1-2, 103--147. 

\bibitem{N} {\sc K.T. Nguyen}, {\it Hypographs satisfying an external sphere condition and the regularity of the minimum time function}, J. Math. Anal. Appl. 372 (2010), no. 2, 611--628.

\bibitem{S} {\sc C. Sinestrari}, Local regularity properties of the minimum time function. Partial differential equation methods in control and shape analysis, 293--308, Lecture Notes in Pure and Appl. Math., 188, Dekker, New York, 199.

\end{thebibliography}
\end{document}